\documentclass[11pt]{amsart}
\usepackage{amssymb,amsthm,amsmath,amstext}
\usepackage{mathrsfs}  % allows nice script font for sheaves
\usepackage{mathdots}
\usepackage{xcolor}
\usepackage{multirow}
\usepackage{enumitem}
\usepackage{colonequals}
\usepackage{cases}
\usepackage{graphicx}
\usepackage{bm}

\usepackage{tikz}
\tikzset{every picture/.style={line width=0.75pt}} %set default line width to 0.75pt  
\usepackage{physics}
\usepackage{amsmath}
\usepackage{tikz}
\usepackage{mathdots}
\usepackage{yhmath}
\usepackage{cancel}
\usepackage{color}
\usepackage{siunitx}
\usepackage{array}
\usepackage{multirow}
\usepackage{amssymb}
\usepackage{gensymb}
\usepackage{tabularx}
\usepackage{extarrows}
\usepackage{booktabs}
\usetikzlibrary{fadings}
\usetikzlibrary{patterns}
\usetikzlibrary{shadows.blur}
\usetikzlibrary{shapes}

\usepackage{subfig}

\usepackage[left=1.0in,top=0.95in,right=1.0in,bottom=0.95in]{geometry}

\usepackage[all]{xy}

\newcommand{\Z}{\mathbb{Z}}

\newcommand{\PP}{\mathbb{P}}

\newcommand{\M}{\mathcal{M}}

\newcommand{\calW}{\mathcal{W}}

\newcommand{\calO}{\mathscr{O}}

\newcommand{\g}[2]{g^{#1}_{#2}}
\newcommand{\BN}[3]{\mathcal{M}^{#2}_{#1,#3}}
\newcommand{\maxk}{\kappa}
\DeclareMathOperator{\Pic}{Pic}
\newcommand{\floor}[1]{\left\lfloor #1 \right\rfloor}
\newcommand{\ceil}[1]{\left\lceil #1 \right\rceil}
\newcommand{\dmax}{d_{max}}

\usepackage[backref=page]{hyperref}
\hypersetup{pdftitle={Maximal Brill-Noether loci via the gonality stratification},pdfauthor={Asher Auel and Richard Haburcak and Hannah Larson}} 
\hypersetup{colorlinks=true,linkcolor=blue,anchorcolor=blue,citecolor=blue}

\usepackage[nameinlink]{cleveref}

\newtheorem{theorem}{Theorem}[section]
\newtheorem{lemma}[theorem]{Lemma}
\newtheorem{prop}[theorem]{Proposition}
\newtheorem{cor}[theorem]{Corollary}

\theoremstyle{definition}
\newtheorem{defn}[theorem]{Definition}

\theoremstyle{definition}
\newtheorem{remark}[theorem]{Remark}

\newtheorem{theoremintro}{Theorem}
\newtheorem{conjectureintro}{Conjecture}

\newcommand{\calG}{\mathcal{G}}

\title{Maximal Brill--Noether loci via the gonality stratification} 

\author{Asher Auel}
\address{Department of Mathematics\\%
	Dartmouth College\\%
	Kemeny Hall\\%
	Hanover, NH 03755}
\email{asher.auel@dartmouth.edu}

\author{Richard Haburcak}
\email{richard.haburcak@dartmouth.edu}

\author{Hannah Larson}
\address{Department of Mathematics\\%
	University of California, Berkeley\\%
	Berkeley, CA 94720}
\email{hlarson@berkeley.edu}

\begin{document}

\thispagestyle{empty}
\vspace*{-.9cm}

\begin{abstract}
We study the restriction of Brill--Noether loci to the gonality
stratification of the moduli space of curves of fixed genus. As an
application, we give new proofs that Brill--Noether loci with
$\rho=-1$ have distinct support, and for fixed $r$ give lower
bounds on when one direction of the non-containments of the Maximal
Brill--Noether Loci Conjecture hold for Brill--Noether loci of rank
$r$ linear systems. Using these techniques, we also show that
Brill--Noether loci corresponding to rank $2$ linear systems are
maximal as soon as $g\ge 28$ and prove the Maximal Brill--Noether Loci
Conjecture for $g=20$.
\end{abstract}

\maketitle

\vspace*{-0.9cm}

\section*{Introduction}
\label{Introduction}

If classical Brill--Noether theory concerns linear systems on general algebraic curves, then \emph{refined Brill--Noether theory} can be viewed as the study of linear systems on special curves.
The main
theorem of Brill--Noether theory \cite{gieseker,griffiths_harris}
implies that the general smooth projective curve $C$ of genus $g$ admits a
nondegenerate morphism $C \to \PP^r$ of degree $d$ if and only if the
\emph{Brill--Noether number}
\[
\rho(g,r,d) \colonequals g - (r+1)(g-d+r)
\]
is non-negative.
A degree $d$ map $C \to \PP^r$ determines a degree $d$ line bundle $L$ on $C$ together with a subspace $V \subset H^0(L)$ of dimension $r + 1$. Such a pair $(L, V)$ is called a $g^r_d$ on $C$.

The last few years have seen a major advance in a refined
Brill--Noether theory for curves of fixed gonality
\cite{cook-powell_jensen,jensen_ranganathan,larson_larson_vogt_2020global,larson_refined_BN_Hurwitz,pflueger}.
In particular, the general smooth projective $k$-gonal curve $C$
of genus $g$ admits a $\g{r}{d}$ if and only if Pflueger's
Brill--Noether number
\[
\rho_k(g,r,d) \colonequals \max_{0 \leq \ell \leq r^\prime} \rho(g,r-\ell,d)-\ell k
\]
where $r^\prime \colonequals \min\{r, g-d+r-1\}$, is non-negative. More broadly, one of
the main goals of refined Brill--Noether theory is to understand when
a ``general'' curve with a $\g{r}{d}$ admits a $\g{s}{e}$, where
here, ``general'' should mean a general curve in a
suitable component of the Brill--Noether locus
\[
\BN{g}{r}{d} = \{ C \in \M_g \; : \; C \text{~admits~a~} \g{r}{d} \}
\]  
when $\rho(g,r,d) < 0$. %This question can also be rephrased in terms of containments between Brill--Noether loci. 
Motivated by conjectures
concerning lifting line bundles on curves in K3 surfaces, the first
two authors posed a conjecture concerning the containments
between Brill--Noether loci.  Adding basepoints and removing
non-basepoints determines various trivial containments between
Brill--Noether loci. Accounting for these, one obtains the notion of the
\emph{expected maximal Brill--Noether loci}, see \Cref{subsection: exp max BN loci}. 

\begin{conjectureintro}[Maximal Brill--Noether Loci Conjecture]
\label{Conj Max BN loci} 
For any $g \geq 3$, except for $g=7,8,9$, the expected maximal
Brill--Noether loci are maximal with respect to containment.
\end{conjectureintro}

\noindent
In other words, the conjecture states that for any two expected
maximal Brill--Noether loci $\BN{g}{r}{d}$ and $\BN{g}{s}{e}$, there
exists a curve $C$ of genus $g$ admitting a $\g{r}{d}$
but not a $\g{s}{e}$, and vice versa.

Using the recently established refined Brill--Noether theory for
curves of fixed gonality, one deduces (see \cite[Proposition
1.6]{auel_haburcak_2022}) that the expected maximal
$\BN{g}{1}{\lfloor\frac{g+1}{2}\rfloor}$ is not contained in any
expected maximal Brill--Noether loci $\BN{g}{r}{d}$ with $r \geq 2$,
and is thus maximal, except when $g=8$.  In this note, we explain how
additional non-containments between expected maximal
Brill--Noether loci can be obtained by restricting to the $k$-gonal
locus.  In particular, we obtain the following.

\begin{theoremintro}
\label{introthm}
Fix $r \geq 2$.  For $g$ sufficiently large, there is a non-containment $\BN{g}{r}{d}\nsubseteq \BN{g}{s}{e}$ for all expected maximal Brill--Noether loci with $s>r$. 
\end{theoremintro}

\noindent In fact, we provide an explicit bound for $g$ in terms of $r$ in \Cref{Theorem one direction of non-containments of max BN loci}.

\medskip 

In \cite{auel_haburcak_2022}, the first two authors proved the Maximal
Brill--Noether Loci Conjecture for $g \leq 19$ and $g=22,23$ using K3
surface techniques. In \Cref{subsec:20}, we show how the techniques
developed to prove \Cref{introthm} by restricting to the $k$-gonal
locus also allow us deduce the following.

\begin{theoremintro}
\label{introthm20}
The Maximal Brill--Noether Loci Conjecture holds for $g=20$.
\end{theoremintro}

Furthermore, we reduce the Maximal Brill--Noether Loci Conjecture
in genus $21$ to verifying just a single non-containment
$\BN{21}{3}{18}\nsubseteq \BN{21}{4}{20}$. There has been a flurry of work on this conjecture by many authors \cite{bud2024brillnoether,bh_2024maximal,bigas2023brillnoether}, showing that it holds in genus $g\leq 23$ and by work of \cite{CHOI2022,CHOI20141458} in genera $g$ such that \[g+1 \text{ or } g+1 \in \left\{ \operatorname{lcm}(1,2,\dots,n) \text{ for some } n \in\mathbb{N}\right\}.\]

\subsection*{Outline} In \Cref{sec: background}, we give background on Brill--Noether loci and Brill--Noether theory of curves of fixed gonality. In \Cref{sec: maximum gonality of BN loci}, we study the maximum gonality stratum contained in a Brill--Noether locus, and show how it can be used to prove non-containments of Brill--Noether loci. In \Cref{sec: distinguishing general BN loci}, we prove further non-containments of Brill--Noether loci. In \Cref{sec: distinguishing maximal BN loci}, we focus on expected maximal Brill--Noether loci, give a new proof that Brill--Noether loci with $\rho=-1$ are distinct, and prove an explicit version of \Cref{introthm} in \Cref{Theorem one direction of non-containments of max BN loci}.

\subsection*{Acknowledgments} The authors would like to thank Xuqiang
Qin for helpful conversations. The first
author received partial support from Simons Foundation grant 712097
and National Science Foundation grant DMS-2200845. The second author would like to
thank the Hausdorff Research Institute for Mathematics funded by the Deutsche Forschungsgemeinschaft (DFG, German Research Foundation) under Germany's Excellence Strategy – EXC-2047/1 – 390685813 for their
generous hospitality during the preparation of this work.  This research was partially conducted during the period the third author served as a Clay Research Fellow.

\section{Brill--Noether loci}\label{sec: background}

Throughout this paper, we work exclusively over the complex numbers, but we note that there are analogous results for the Brill--Noether theory for curves of fixed gonality in positive characteristic.

\subsection{Brill--Noether loci}
Brill--Noether theory studies maps of curves $C$ to projective space. A
nondegenerate morphism $C \to \PP^r$ of degree $d$ is determined by a
$\g{r}{d}$, namely, a point in the space 
\[
G^r_d(C) \colonequals \{(L,V) \;\mid\; L\in\Pic^d(C),~ V\subseteq
H^0(C,L),~ \dim V=r+1\}.
\] 
The image of the natural map $G^r_d(C)\to\Pic^d(C)$ is
\[
W^r_d(C)\colonequals \{L\in\Pic^d(C) \;\mid\; h^0(C,L)\ge r+1\}.
\]
These spaces can be globalized to moduli spaces $\calG^r_d\to \M_g$ and $\calW^r_d\to\M_g$ over the moduli space $\M_g$ of
smooth curves of genus $g$, where the fiber above $C$ is $G^r_d(C)$
and $W^r_d(C)$, respectively. The Brill--Noether loci
\[
\BN{g}{r}{d} \colonequals \{C\in\M_g \mid C \text{ admits a
$\g{r}{d}$}\}
\] 
are the images of the corresponding maps $\calG^r_d\to\M_g$.

The Brill--Noether--Petri theorem
\cite{griffiths_harris,lazarsfeld:Brill-Noether_without_degenerations}
states that for a general curve $C$ of genus $g$, the variety $W^r_d(C)$ is
non-empty exactly when the Brill--Noether
number \[\rho(g,r,d)\colonequals g-(r+1)(g-d+r)\] is non-negative. Consequently, when
$\rho(g,r,d)\ge 0$, we have $\BN{g}{r}{d}=\M_g$.
Meanwhile, when $\rho(g,r,d)<0$,
$\BN{g}{r}{d}$ is a proper subvariety of $\M_g$, all of whose
components have codimension at most $-\rho(g,r,d)$ \cite{steffen_1998}. It is known that
Brill--Noether loci with $-3 \le \rho(g,r,d)\le -1$ have codimension
exactly $-\rho$, and Brill--Noether loci with $\rho=-1$ and $\BN{g}{2}{d}$ with $\rho=-2$ are
irreducible \cite{CHOI2022,Eisenbud_Harris_1989,steffen_1998}.

The stratification of $\M_g$ by Brill--Noether loci and the
interaction of various Brill--Noether loci is useful in the study of
the birational geometry of $\M_g$, see \cite{Farkas2000,
harris_mumford}.  Brill--Noether loci with $\rho=-1$ have been studied
by Harris, Mumford, Eisenbud, and Farkas
\cite{eisenbud_harris,Eisenbud_Harris_1989,Farkas2000,Farkas_2001,harris_mumford},
in particular, in the study of the Kodaira dimension of
$\M_{23}$. More recently, Choi, Kim, and
Kim~\cite{CHOI2012377,CHOI20141458} showed in a series of papers that
Brill--Noether divisors have distinct support. Choi and
Kim~\cite{CHOI2022} showed that Brill--Noether loci with $\rho=-2$ are not contained in each other, and that $\BN{g}{2}{d}$ with $\rho=-2$ are irreducible; and further showed
that Brill--Noether loci with $\rho=-2$ are not contained in certain
Brill--Noether divisors; for new proofs of these non-containments see
\Cref{theorem BN loci with rho=-1 are distinct} and
\Cref{theorem rho=-2 loci not contained in BN divisors improvement}.

\subsection{Expected maximal Brill--Noether loci}\label{subsection: exp max BN loci}

There are various containments among Brill--Noether loci. In
particular, there are \emph{trivial containments}
$\BN{g}{r}{d}\subseteq \BN{g}{r}{d+1}$ obtained by adding a basepoint
to a $\g{r}{d}$ on $C$; and $\BN{g}{r}{d}\subseteq \BN{g}{r-1}{d-1}$
when $\rho(g,r-1,d-1)<0$ by subtracting a non-basepoint, cf.\
\cite{Farkas2000,Lelli-Chiesa_the_gieseker_petri_divisor_g_le_13}. Modulo
these trivial containments, the first two
authors~\cite{auel_haburcak_2022} defined the \emph{expected maximal
  Brill--Noether loci} as the $\BN{g}{r}{d}$ where for fixed rank
$r\ge 1$, satisfying $2r\le d\le g-1$, the degree $d$ is maximal such
that $\rho(g,r,d)<0$ and $\rho(g,r-1,d-1)\ge 0$. Note that (after
accounting for Serre duality, which gives
$\BN{g}{r}{d}=\BN{g}{g-d+r-1}{2g-2-d}$) every Brill--Noether locus
with $\rho(g,r,d)<0$ is contained in at least one expected maximal
Brill--Noether locus. They then posed \Cref{Conj Max BN loci}, which
says that the expected maximal Brill--Noether loci should be maximal
with respect to containment, except when $g = 7, 8, 9$.  (In genus
$7, 8,$ and $9$, there are unexpected containments of Brill--Noether
loci coming from projections from points of multiplicity $\ge 2$ in
genus $7$ and $9$ \cite[Propositions 6.2 and 6.4]{auel_haburcak_2022}
or from a trisecant line in genus $8$, as shown by Mukai \cite[Lemma
3.8]{Mukai_Curves_and_grassmannians_1993}.)

Let $\gamma(r, d) := d - 2r$ be the Clifford index.  

\begin{lemma} \label{rbound}
For an expected maximal Brill--Noether locus $\M^r_{g,d}$, we have
$r\leq \ceil{\sqrt{g}-1}$. More precisely, an expected maximal
Brill--Noether locus $\BN{g}{r}{d}$ exists for some $d$ if and only if
\begin{equation} \label{rcases} 1\leq r \leq \begin{cases}
	\ceil{\sqrt{g}-1} & \text{if~} g\geq\floor{\sqrt{g}}^2+\floor{\sqrt{g}}\\
	\floor{\sqrt{g}-1} & \text{if~} g<\floor{\sqrt{g}}^2+\floor{\sqrt{g}}.
\end{cases}
\end{equation}
\end{lemma}

\begin{proof}
First observe that if $\BN{g}{r}{d}$ is expected maximal and $\rho(g,
r', d') < 0$ with $d < d' \leq g - 1$, then $r < r'$. Indeed, since
$\BN{g}{r}{d}$ is expected maximal, we must have $\rho(g, r, d') \geq
0$. But $\rho(g, r', d') < 0$, hence we must have $r' > r$, since
$\rho(g, r, d')$ is decreasing as a function of $r$. Thus, if
$\BN{g}{r}{d}$ is expected maximal then $r \leq r'$ where $r'$ is such
that $\BN{g}{r'}{g-1}$ or $\BN{g}{r^\prime}{g-2}$ is expected maximal (in case there is a trivial containment $\BN{g}{r^\prime +1}{g-1}\subseteq \BN{g}{r'}{g-2}$).

As observed in \cite[Remark 1.2]{auel_haburcak_2022}, the maximum in the $(r,\gamma)$-plane of the graph of $\rho(g,r,d)=0$ occurs at $r=\sqrt{g}-1$, the intersection of $\rho(g,r,d)=0$ and $d=g-1$. As the line $d=g-1$ has slope $-2$, the trivial containments occur with slope $\infty$ or with slope $-1$, and $\rho(g,r,d)$ is an increasing function of $d$, one can expect that the maximum $r$ of expected maximal loci occurs when $r\approx \sqrt{g}-1$. We show that for expected maximal loci $r\leq \ceil{\sqrt{g}-1}$, and in some cases there is a trivial containment giving $r\leq \floor{\sqrt{g}-1}$.

We compute that $\rho(g,\ceil{\sqrt{g}-1},g-1) = g - \ceil{\sqrt{g}}^2
< 0$ unless $g$ is a square.  When $g$ is a square, we see that
$\rho(g,\ceil{\sqrt{g}-1},g-2) = g -
\sqrt{g}(\sqrt{g}+1) < 0$ and that
$\rho(g,\ceil{\sqrt{g}-1}-1,g-3) = g - (\sqrt{g})^2 + 1 = 1 > 0$,
hence the Brill--Noether locus $\BN{g}{\sqrt{g}}{g-1}$ is trivially
contained in $\BN{g}{\sqrt{g}-1}{g-2}$, and the latter is expected
maximal.  Note that when $g$ is square, $g<\floor{\sqrt{g}}^2+\floor{\sqrt{g}}$.

So we assume that $g$ is not a square.   
Noting that
$\rho(g,\ceil{\sqrt{g}-1},g) = \rho(g,\ceil{\sqrt{g}-1}-1,g-2) = g - \ceil{\sqrt{g}}(\ceil{\sqrt{g}}-1)
= g - \floor{\sqrt{g}}(\floor{\sqrt{g}}+1)$,
we see that $\BN{g}{\ceil{\sqrt{g}-1}}{g-1}$ is expected maximal if
and only if $g \geq \floor{\sqrt{g}}^2 + \floor{\sqrt{g}}$.

When $g < \floor{\sqrt{g}}^2 + \floor{\sqrt{g}}$, we see that
$\rho(g,\floor{\sqrt{g}-1},g-2) = g- \floor{\sqrt{g}}^2 +
\floor{\sqrt{g}} < 0$ and that
$\rho(g,\floor{\sqrt{g}-1}-1,g-3) = g - \floor{\sqrt{g}}^2 + 1 > 0$.
Hence the Brill--Noether locus $\BN{g}{\ceil{\sqrt{g}-1}}{g-1}$ is
trivially contained in the Brill--Noether locus
$\BN{g}{\floor{\sqrt{g}-1}}{g-2}$, and the latter is expected maximal.
\end{proof}

Once a rank $r$ satisfying the conditions of \Cref{rbound} is fixed,
the degree $d$ that makes $\M^r_{g,d}$ expected maximal is uniquely
determined: it is the largest $d$ such that $\rho(g, r, d) < 0$, namely 
\begin{equation} \label{dmaxdef}
d = \dmax(g, r) \colonequals  r+\left\lceil \frac{gr}{r+1} \right\rceil -1.
\end{equation}
For ease of notation, for each $r$ satsifying \eqref{rcases}, we shall write $\M^r_g \colonequals
\M^r_{g, \dmax(g,r)}$ for the expected maximal Brill--Noether locus of
rank $r$ linear series. In other words, \Cref{rbound} says that the
expected maximal Brill--Noether loci in $\M_g$ are precisely the $\M^r_g$ for $r$ satisfying \eqref{rcases}.

\subsection{Brill--Noether theory of curves with fixed gonality}
Recall that the \emph{gonality} of a curve is the minimal $k$ such
that $C$ admits a $\g{1}{k}$.  The Brill--Noether locus $\BN{g}{1}{k}$
is the closure of the locus of $k$-gonal curves. Because the
corresponding Hurwitz space of degree $k$ covers is irreducible,
$\BN{g}{1}{k}$ is irreducible. It therefore makes sense to talk about
a general $k$-gonal curve.

In general, $W^r_d(C)$ can have multiple components of varying dimensions.
Pflueger~\cite{pflueger} showed that for a general $k$-gonal curve
\begin{equation} \label{rhok}
\dim W^r_d(C)\le \rho_k(g,r,d)\colonequals
\max_{\ell\in\{0,\dots,r^\prime\}} \rho(g,r-\ell,d)-\ell k,
\end{equation}
where $r^\prime \colonequals \min\{r, g-d+r-1\}$. 
Since $W^r_d(C)$ may
not have pure dimension, $\dim W^r_d(C)$ above means the maximum of
the dimensions of its components.
Subsequently, Jensen
and Ranganathan \cite{jensen_ranganathan} showed that a component of the maximum possible dimension
exists.
\begin{theorem}[Jensen--Ranganathan \cite{jensen_ranganathan}] \label{JR}
If $C$ is a general $k$-gonal curve, then $\dim W^r_d(C) = \rho_k(g, r, d)$. In particular, a general $k$-gonal curve admits a $g^r_d$ if and only if $\rho_k(g, r, d) \geq 0$.
\end{theorem}

The dimensions and enumeration of all components of $W^r_d(C)$ were subsequently determined by the third author and others by studying associated splitting loci
\cite{Powell_jensen_fixed_gonality,cook-powell_jensen,jensen_ranganathan,larson_larson_vogt_2020global,larson_refined_BN_Hurwitz}. For a summary of these results, see \cite{jensen_payne_2021recent}. Our applications to maximal Brill--Noether loci will rely only on the statement in \Cref{JR}.

\section{The maximal gonality stratum in a Brill--Noether locus}
\label{sec: maximum gonality of BN loci}

Throughout the remainder of this paper, we add the assumption that $\rho<0$ for a Brill--Noether locus. Our main new ingredient is the following invariant of a Brill--Noether
locus.

\begin{defn}
For a given genus $g$, rank $r$, and degree $d$, we define
$\maxk(g,r,d)$ to be the maximal $k \geq 1$ such that a general curve
of genus $g$ and gonality $k$ admits a $\g{r}{d}$. In other words,
$\maxk(g,r,d)$ is the maximal $k$ such that $\M_{g,k}^1 \subseteq
\M^r_{g,d}$.
\end{defn}

Our basic observation is that $\maxk$ can separate Brill--Noether loci. 

\begin{prop}
\label{Theorem makx>maxk' implies non-containment}
If $\maxk(g,r,d)>\maxk(g,s, e)$, then $\BN{g}{r}{d}\nsubseteq \BN{g}{s}{e}$.
\end{prop}
\begin{proof}
Assume, to get a contradiction, that $\BN{g}{r}{d}\subseteq
\BN{g}{s}{e}$, and let
$k=\maxk(g,r,d)>\maxk(g,s, e)$. By the definition of $\maxk$, $\BN{g}{1}{k}\nsubseteq\BN{g}{s}{e}$. But the
assumption implies that $\BN{g}{1}{k}\subseteq \BN{g}{r}{d}\subseteq
\BN{g}{s}{e}$, which is a contradiction.
\end{proof}

\begin{center}
\begin{tikzpicture}[x=0.75pt,y=0.75pt,yscale=-1,xscale=1]
	%uncomment if require: \path (0,360); %set diagram left start at 0, and has height of 360
	
	%Shape: Ellipse [id:dp6006131312340179] 
	\draw  [color={rgb, 255:red, 140; green, 140; blue, 140 }  ,draw opacity=1 ][fill={rgb, 255:red, 140; green, 140; blue, 140 }  ,fill opacity=0.15 ] (96.28,110.17) .. controls (136.21,89.83) and (215.48,97.11) .. (273.33,126.42) .. controls (331.18,155.73) and (345.7,195.98) .. (305.77,216.32) .. controls (265.84,236.65) and (186.57,229.38) .. (128.72,200.07) .. controls (70.87,170.76) and (56.35,130.51) .. (96.28,110.17) -- cycle ;
	%Flowchart: Punched Tape [id:dp8984045705635415] 
	\draw  [color={rgb, 255:red, 208; green, 2; blue, 27 }  ,draw opacity=1 ][fill={rgb, 255:red, 208; green, 2; blue, 27 }  ,fill opacity=0.27 ][line width=1.5]  (98.42,72.56) .. controls (99.83,78.01) and (116.73,82.66) .. (136.18,82.96) .. controls (155.62,83.26) and (170.24,79.09) .. (168.83,73.64) .. controls (167.42,68.2) and (182.04,64.02) .. (201.48,64.32) .. controls (220.92,64.62) and (237.83,69.28) .. (239.24,74.72) -- (259.65,153.62) .. controls (258.24,148.17) and (241.34,143.51) .. (221.9,143.21) .. controls (202.45,142.91) and (187.84,147.09) .. (189.25,152.53) .. controls (190.65,157.98) and (176.04,162.15) .. (156.59,161.85) .. controls (137.15,161.56) and (120.25,156.9) .. (118.84,151.45) -- cycle ;
	%Curve Lines [id:da16813687571110747] 
	\draw [color={rgb, 255:red, 84; green, 84; blue, 84 }  ,draw opacity=1 ][line width=3]    (118.82,151.45) .. controls (128.31,161.31) and (145.51,161.29) .. (156.59,161.85) .. controls (167.67,162.42) and (190.13,158.74) .. (189.24,152.53) .. controls (188.34,146.33) and (201.7,144.56) .. (212.57,143.28) .. controls (223.45,141.99) and (253.69,145.37) .. (259.66,153.62) ;
	%Curve Lines [id:da10220716374996619] 
	\draw [color={rgb, 255:red, 0; green, 0; blue, 255 }  ,draw opacity=1 ][line width=1.5]    (111.78,172.91) .. controls (145.77,176.77) and (140.66,99.45) .. (150.63,100.76) .. controls (160.6,102.07) and (195.56,189.65) .. (207.7,169.04) .. controls (219.84,148.43) and (182.2,102.05) .. (211.34,93.03) .. controls (240.49,84.01) and (246.56,127.82) .. (269.63,93.03) ;
	%Straight Lines [id:da35695927322476884] 
	\draw    (168,185.5) -- (182.79,166.09) ;
	\draw [shift={(184,164.5)}, rotate = 127.3] [color={rgb, 255:red, 0; green, 0; blue, 0 }  ][line width=0.75]    (10.93,-3.29) .. controls (6.95,-1.4) and (3.31,-0.3) .. (0,0) .. controls (3.31,0.3) and (6.95,1.4) .. (10.93,3.29)   ;
	%Straight Lines [id:da816840239028813] 
	\draw    (217,182.5) -- (233.06,152.27) ;
	\draw [shift={(234,150.5)}, rotate = 117.98] [color={rgb, 255:red, 0; green, 0; blue, 0 }  ][line width=0.75]    (10.93,-3.29) .. controls (6.95,-1.4) and (3.31,-0.3) .. (0,0) .. controls (3.31,0.3) and (6.95,1.4) .. (10.93,3.29)   ;
	%Shape: Ellipse [id:dp8997673299940923] 
	\draw  [fill={rgb, 255:red, 0; green, 0; blue, 0 }  ,fill opacity=1 ] (182.95,156.76) .. controls (182.95,154.27) and (184.85,152.25) .. (187.2,152.25) .. controls (189.55,152.25) and (191.45,154.27) .. (191.45,156.76) .. controls (191.45,159.25) and (189.55,161.27) .. (187.2,161.27) .. controls (184.85,161.27) and (182.95,159.25) .. (182.95,156.76) -- cycle ;
	
	% Text Node
	\draw (71.46,20.97) node [anchor=north west][inner sep=0.75pt]   [align=left] {\textcolor[rgb]{0.82,0.01,0.11}{$\displaystyle \kappa $}\textcolor[rgb]{0.82,0.01,0.11}{(}\textcolor[rgb]{0.82,0.01,0.11}{$\displaystyle g,r,d$}\textcolor[rgb]{0.82,0.01,0.11}{)}$\displaystyle =k >\textcolor[rgb]{0,0,1}{\operatorname{\kappa }}\textcolor[rgb]{0,0,1}{(}\textcolor[rgb]{0,0,1}{g,s,e}\textcolor[rgb]{0,0,1}{)} =k-1$};
	% Text Node
	\draw (137,185.4) node [anchor=north west][inner sep=0.75pt]  [font=\small]  {$\mathcal{M}_{g,k-1}^{1} \subset \mathcal{\textcolor[rgb]{0.33,0.33,0.33}{M}}\textcolor[rgb]{0.33,0.33,0.33}{_{g,k}^{1}} \subset \mathcal{\textcolor[rgb]{0.49,0.49,0.49}{M}}\textcolor[rgb]{0.49,0.49,0.49}{_{g,k+1}^{1}}$};
	% Text Node
	\draw (60.31,71.65) node [anchor=north west][inner sep=0.75pt]    {$\mathcal{\textcolor[rgb]{0.82,0.01,0.11}{M}}\textcolor[rgb]{0.82,0.01,0.11}{_{g,d}^{r}}$};
	% Text Node
	\draw (272.24,73.7) node [anchor=north west][inner sep=0.75pt]    {$\mathcal{\textcolor[rgb]{0,0,1}{M}}\textcolor[rgb]{0,0,1}{_{g,e}^{s}}$};

\end{tikzpicture}
\\
A general curve of gonality $k$ is contained in \textcolor{red}{$\BN{g}{r}{d}$}, but not in \textcolor{blue}{$\BN{g}{s}{e}$}.
\end{center}

\begin{remark}
Noting the trivial containments of Brill--Noether loci, if
$\maxk(g,r,d)>\maxk(g,s, e)$ then \Cref{Theorem makx>maxk' implies non-containment} in fact implies non-containments of Brill--Noether loci  of the form $\BN{g}{r}{d}\nsubseteq \BN{g}{s}{a}$ for all $a\le e$ and of the form $\BN{g}{r}{d}\nsubseteq\BN{g}{s+i}{e+i}$ for all $i\ge 1$.
\end{remark}

By \Cref{JR},
a general curve of gonality $k$ admits a $g^r_d$ if and only if $\rho_k(g, r, d) \geq 0$, so
\begin{equation} \label{maxk1}
\maxk(g,r,d) = \max\{k \; : \; \rho_k(g,r,d) \geq 0 \}.
\end{equation}
We remark that if $\M^r_{g,d}$ is non-empty, then $d - 2r \geq 0$ by
Clifford's theorem, from which we can deduce the following bound.

\begin{lemma}
Let $g,r,d \geq 1$ satisfy $d - 2r \geq d$ and $g-d+r\ge 1$.  Then $\maxk(g,r,d)
\geq 2$. 
\end{lemma}
\begin{proof}
If we show that $\M^1_{g,2}$ is contained in every non-empty
Brill--Noether locus $\M^r_{g,d}$ then it follows that $\maxk(g, r, d)
\geq 2$.  To this end, if $C$ has a $g^1_2$, then for any $p \in C$ we
have that $r g^1_2 + (d - 2r)p$ is a $g^r_d$ on $C$.

We can also argue directly with Pflueger's formula \eqref{rhok} by
setting $\ell=\min\{r,g-d+r-1\}$ and $k=2$, from which we obtain
$\rho_2(g,r,d) \geq d-2r\ge 0$.   Then  $\maxk(g, r, d)
\geq 2$ by equation \eqref{maxk1}. 
\end{proof}

Despite the combinatorial nature of \eqref{rhok} and \eqref{maxk1},
we have the following
closed formula.

\begin{prop}\label{prop formula for maxk}
	Suppose $d\le g-1$. We have\[\maxk(g,r,d)=\begin{cases}
		\floor{\frac{d}{r}} & \text{ if } g+1 > \floor{\frac{d}{r}}+d \\
		g+1-\gamma(r,d)+\floor{-2\sqrt{-\rho(g,r,d)}} & \text{ else. }
	\end{cases}\]
\end{prop}

\begin{proof}
Note that $d\le g-1$ is equivalent to $r=\min \{r, g-d+r-1\}$, hence $r=r^\prime$. For fixed $g,r,d$, we observe that $\rho_k(g,r,d)$ is a non-increasing function of $k$, as can be seen by writing 
\[
\rho_k(g,r,d)=\max_{\ell\in\{0,\dots,r\}} \rho(g,r-\ell,d)-\ell k = \max_{\ell\in\{0,\dots,r\}}\rho(g,r,d)+\left(g-k-\gamma(r,d)+1\right)\ell -\ell^2.
\] 
In particular, $\rho_k(g,r,d)$ is a maximum over values of a concave down parabola. The maximum of this parabola (over all real values of $\ell$) is attained at
\[
\ell^\ast\colonequals\frac{g-k-\gamma(r,d)+1}{2}.
\]
Thus the maximum of the parabola over our range of integers occurs at $\ell = \lceil \ell^\ast \rceil$ if $0\le \ell^\ast \le r$. Otherwise the maximum of is attained at $\ell=0$ (if $\ell^* < 0$) or at $\ell=r$ (if $\ell^* > r$).

We now treat each of the two cases in the statement. First suppose $g + 1 > \lfloor \frac{d}{r} \rfloor + d$. If $k = \lfloor \frac{d}{r} \rfloor$, then $k < g + 1 - d$, so $\ell^* > r$ and one checks $\rho_k(g, r, d) = \rho(g, 0, d) - rk = d - rk \geq 0$. Meanwhile, if $k = \lfloor \frac{d}{r} \rfloor +1$, then $k \leq g + 1 - d$, so $\ell^* \geq r$. Hence, $\rho_k(g, r, d) = \rho(g, 0, d) -rk = d - rk < 0$. Since $\rho_k(g, r, d)$ is non-increasing, it follows that $\maxk(g, r, d) = \lfloor \frac{d}{r} \rfloor$.

Now suppose $g + 1 \le \lfloor \frac{d}{r} \rfloor + d$. In this case, we can bound
\begin{align*}
-\rho(g, r, d) &= (r+1)(g - d + r) - g =r(g-d)-d+r^2+r\\
&\leq r\left(\floor{\frac{d}{r}}-1\right)-d+r^2+r = r\floor{\frac{d}{r}}-d+r^2\\
&= d-(d\bmod r) -d+r^2 \le r^2
\end{align*}
where $(d\bmod r)$ denotes the remainder after dividing $d$ by $r$ and
where we remark the identity  $r\floor{\frac{d}{r}} = d-(d\bmod r)$.
Thus $\sqrt{-\rho(g, r, d)} \leq  r$ and it follows that the claimed value for $\maxk(g, r, d)$ lies in the range $k \geq g + 1 - d$. If $k \geq g + 1 - d$, then
$\ell^* \leq r$, so 
\[\rho_k(g, r, d) = \rho(g, r, d) + 2\ell^*\lceil \ell^* \rceil - \lceil \ell^* \rceil^2.\]
If $\ell^*$ is an integer then $\rho_k(g, r, d) \geq 0$ is equivalent to $(\ell^*)^2 \geq - \rho(g, r, d)$, which in turn is equivalent to
\begin{equation*}
k\le g+1-\gamma(r,d)-2\sqrt{-\rho(g,r,d)}.
\end{equation*}
Otherwise $\lceil \ell^* \rceil = \ell^* + \frac{1}{2}$, so $\rho_k(g, r, d) \geq 0$ is equivalent to
\[\rho(g, r, d) + 2\ell^*(\ell^* + \tfrac{1}{2}) - (\ell^* + \tfrac{1}{2})^2 \geq 0\]
which in turn is equivalent to $(\ell^*)^2 \geq - \rho(g, r, d) - \tfrac{1}{4}$. In this case, we obtain the bound
\[k \leq  g+1-\gamma(r,d)-2\sqrt{-\rho(g,r,d)-\tfrac{1}{4}}. \] The result now follows from \Cref{lemma -1/4 doesn't matter in maxk bound} below.
\end{proof}

\begin{lemma}\label{lemma -1/4 doesn't matter in maxk bound}
	For any integer $n> 0$, we have $\floor{-2\sqrt{n}}=\floor{-2\sqrt{n-\frac{1}{4}}}$.
\end{lemma}
\begin{proof}
We see that $\ceil{2\sqrt{n-\frac{1}{4}}}\le
\ceil{2\sqrt{n}}$. Suppose they are not equal. Then there is an
$m>0$ such that $2\sqrt{n-\frac{1}{4}}\le m <
2\sqrt{n}$. Squaring the inequalities gives $4n-1\le m^2<4n$, whereby
$m^2=4n-1$. However, since $m^2\equiv0,1\bmod 4$, we arrive at a contradiction.
\end{proof}

\subsection{Genus \texorpdfstring{$20$ and $21$}{}}
\label{subsec:20}

Using \Cref{Theorem makx>maxk' implies non-containment}, we prove
\Cref{Conj Max BN loci} in genus $20$ and reduce the genus $21$ case
to a single non-containment.

In genus $20$, the expected maximal Brill--Noether loci are $\BN{20}{1}{10}$, $\BN{20}{2}{15}$, $\BN{20}{3}{17}$, $\BN{20}{4}{19}$.

\begin{theorem}
	The Maximal Brill--Noether Loci Conjecture, \Cref{Conj Max BN loci}, holds in genus $20$.
\end{theorem}
\begin{proof}
	In \cite{auel_haburcak_2022}, \Cref{Conj Max BN loci} for $g=20$ was reduced to proving $\BN{20}{3}{17}\nsubseteq\BN{20}{4}{19}$. We compute \[\maxk(20,3 , 17)= 6> 5 =\maxk(20,4 , 19),\] whereby \Cref{Theorem makx>maxk' implies non-containment} gives the desired non-containment.
\end{proof}

In genus $21$, the expected maximal Brill--Noether loci are $\BN{21}{1}{11}$, $\BN{21}{2}{15}$, $\BN{21}{3}{18}$, $\BN{21}{4}{20}$. We summarize the known non-containments without proof, as they follow directly from \cite{auel_haburcak_2022} and \Cref{Theorem makx>maxk' implies non-containment}.

\begin{theorem}
	In genus $21$, the loci $\BN{21}{1}{11}$, $\BN{21}{2}{15}$, $\BN{21}{4}{20}$ are maximal. There are also non-containments
	\begin{itemize}
		\item $\BN{21}{3}{18}\nsubseteq\BN{21}{1}{11}$ and
		\item $\BN{21}{3}{18}\nsubseteq\BN{21}{2}{15}$.
	\end{itemize}
\end{theorem}

\begin{remark}
	To verify that \Cref{Conj Max BN loci} holds in genus $21$ the only remaining non-containment is $\BN{21}{3}{18}\nsubseteq \BN{21}{4}{20}$
\end{remark}

\section{Applications to Brill--Noether loci}\label{sec: distinguishing general BN loci}

We apply \Cref{Theorem makx>maxk' implies non-containment} to prove new
non-containments between Brill--Noether loci.
We first collect a few observations about $\rho$ and $\gamma$ for
Brill--Noether loci. As $\maxk$ depends explicitly on $\gamma$ and
$\rho$, it is natural to ask if $\gamma$ and $\rho$ are sufficient to
numerically identify a $\g{r}{d}$.

\begin{prop}
Let $g,r,d,s, e$ be positive integers. If $\rho(g,r,d)=\rho(g,s, e)$ and $\gamma(r,d)=\gamma(s, e)$, then either
\begin{enumerate}[label=(\roman*)]
	\item $r=s$ and $d=e$, or
	\item $s = g-d+r-1$ and $e = 2g-2-d$.
\end{enumerate}

\end{prop}
\begin{proof}
Since
$\gamma(r,d)=\gamma(s, e)$, writing $s=r+\delta$
gives $e = d+2\delta$. Simplifying the expression for $\rho$,
we find
\[
\rho(g,r,d)=\rho(g,r+\delta,d+2\delta) = \rho(g,r,d)+\delta(d-g+1)+\delta^2.
\]
Hence we find that either $\delta=0$ and $(i)$ holds, or $\delta=g-d-1$ and $(ii)$ holds. 
\end{proof}

\begin{remark}
	Thus two complete linear systems of type $\g{r}{d}$ and $\g{s}{e}$ are of the same type or of Serre dual type (numerically, $\g{s}{e}=K_C-\g{r}{d}$) if and only if $\rho(g,r,d)=\rho(g,s,e)$ and $\gamma(r,d)=\gamma(s, e)$. In particular, distinct Brill--Noether loci with the same $\rho$ will
	not have the same $\gamma$, and vice versa.  
\end{remark}

For Brill--Noether loci with the same
$\rho$, \Cref{Theorem makx>maxk' implies non-containment} easily
gives one non-containment. A similar result was recently proved by Teixidor i Bigas in \cite{bigas2023brillnoether}.

\begin{cor}\label{cor distinguishing loci with same rho}
Suppose $\rho(g,s, e)=\rho(g,r,d)$ and $g+1\le \floor{\frac{d}{r}}+d, \floor{\frac{e}{s}}+e$. If $\gamma
(r,d)<\gamma(s, e)$, then $\BN{g}{r}{d}\nsubseteq
\BN{g}{s}{e}$.
\end{cor}

As the expected codimension of a Brill--Noether locus $\BN{g}{r}{d}$ in $\M_g$ is $-\rho(g,r,d)$, one expects non-containments of the form $\BN{g}{r}{d}\nsubseteq \BN{g}{s}{e}$ when $\rho(g,r,d)>\rho(g,s, e)$. Thus it is interesting to find non-containments of Brill--Noether loci in the other direction. We give a general statement on when a Brill--Noether locus is not contained in Brill--Noether divisors (loci with $\rho=-1$).

\begin{prop}\label{prop BN locus contianed in BN divisor}
	Suppose that $\rho(g,s,e)=-1$, $e-2s>d-2r+\left\lceil 2\sqrt{-\rho(g,r,d)}\right\rceil-2$, and $g+1\le \floor{\frac{d}{r}}+d$, then $\BN{g}{r}{d}\nsubseteq\BN{g}{s}{e}$. 
\end{prop}
\begin{proof}
	Note that if $s=1$, then $\rho(g,s,e)=-1$ implies $g+1\le \floor{\frac{e}{s}}+e$. Simple computations show that when $\rho(g,s,e)=-1$, the condition $\maxk(g,r,d)>\maxk(g,s,e)$ is equivalent to the condition $\gamma(s,e)>\gamma(r,d)+\left\lceil 2\sqrt{-\rho(g,r,d)}\right\rceil-2$.
\end{proof}

\section{Applications to expected maximal Brill--Noether loci}\label{sec: distinguishing maximal BN loci}

\subsection{Formulas for expected maximal loci}
For expected maximal Brill--Noether loci, one can make the formulas for $\rho$ and $\maxk$ more explicit. Given $g$ and $r$ with $r \leq \sqrt{g} - 1$, recall that we write $\dmax(g, r)$ for the degree $d$ so that $\M^r_{g,d}$ is expected maximal, given in \eqref{dmaxdef}.
\begin{lemma} \label{rholem}
	Let $g\bmod r+1$ be the non-negative representative. For an expected maximal Brill--Noether locus $\BN{g}{r}{d}$, we have $-\rho(g,r,d)=r+1-(g\bmod r+1)$.
\end{lemma}
\begin{proof}
	We compute 
	\begin{align*}
		\rho(g,r,d_{max}(g,r))&=g-(r+1)(g-d_{max}(g,r)+r)\\
		&= -gr-r^2-r+(r+1)\left( r-1+\left\lceil \frac{gr}{r+1}\right\rceil \right)\\
		&=-gr-r-1+(r+1)\left\lceil \frac{gr}{r+1}\right\rceil.
	\end{align*}
Recalling the identity $y\left\lfloor \frac{x}{y}\right\rfloor =x-
(x\bmod y)$ for integers $x$ and $y > 0$, we see that
\[(r+1)\left\lfloor \frac{-gr}{r+1}\right\rfloor=-gr-(-gr\bmod r+1)=-gr-(g\bmod r+1).\]
	Thus 
	\begin{align*}
		-\rho(g,r,d_{max}(g,r))&=r+gr+1+(r+1)\left\lfloor \frac{-gr}{r+1}\right\rfloor\\
		&=r+gr+1 -gr -(g\bmod r+1)\\
		&= r+1 - (g\bmod r+1). \qedhere
	\end{align*}
\end{proof}

With this formula for $-\rho$, we can simplify our formula for $\maxk$ for expected maximal loci.
\begin{prop} \label{maxformula}
	For an expected maximal Brill--Noether locus $\BN{g}{r}{d}$ with $r\ge 2$, we have \[\maxk(g,r,d)=g+r+2+\left\lfloor \frac{-gr}{r+1} \right\rfloor+ \left\lfloor-2\sqrt{r+1 -(g\bmod(r+1))} \right\rfloor.\]
\end{prop}

\begin{proof}
We claim that if $r \geq 2$, and $d =\dmax(g, r)$, then  $g + 1 \leq \lfloor \frac{d}{r} \rfloor + d$. Once this is established, combining \Cref{rholem} and \Cref{prop formula for maxk} and substituting $d = \dmax(g, r)$ gives the result.
	To prove the claim, we let $d=\dmax(g,r) = r-1+\lceil \frac{rg}{r+1} \rceil$ and expand
	\begin{align*}
	\floor{\frac{d}{r}} + d &= \floor{1-\frac{1}{r}+\frac{1}{r}\ceil{\frac{rg}{r+1}}} + r - 1 +\ceil{\frac{rg}{r+1}} \\
	&> r-\frac{1}{r} - 1 +\frac{1}{r}\ceil{\frac{rg}{r+1}} +\ceil{\frac{rg}{r+1}} \\
	&\geq r - \frac{1}{r} -1 + g
        \geq \frac{1}{2} + g.
	\end{align*}
Above, we have used that $r - \frac{1}{r}$ is increasing for $r \geq \frac{1}{2}$, so the assumption $r \geq 2$ means $r - \frac{1}{r} \geq 2 - \frac{1}{2}$. Thus, we have $\floor{\frac{d}{r}} + d > \frac{1}{2} + g$. Since the left-hand side is an integer, the claim follows.
\end{proof}

\subsection{Non-containments of Brill--Noether loci with \texorpdfstring{$\rho=-1,-2$}{}}\label{subsection noncontainments of BN loci with small rho}

One can easily check that Brill--Noether loci with $\rho=-1,-2$ are
expected maximal. Indeed, $\rho(g,r,d+1)=\rho(g,r,d)+r+1$, and
$\rho(g,r-1,d-1)=\rho(g,r,d)+g-d+r$. As shown in
\cite{Eisenbud_Harris_1989,steffen_1998}, Brill--Noether loci
with $\rho=-1$ are irreducible. Thus to show that such
Brill--Noether loci have distinct support, it suffices to show one
non-containment. Choi, Kim, and Kim~\cite{CHOI2022,CHOI20141458}
prove various
non-containments of the form $\BN{g}{r}{d}\nsubseteq\BN{g}{s}{e}$ when
$\rho(g,r,d)=-2$ and $\rho(g,s,e)=-1$. We provide new proofs of these
non-containments using $\maxk$. We note that Choi, Kim, and Kim also show both non-containments for loci with $\rho=-2$, which we do not prove here.

\begin{theorem}\label{theorem BN loci with rho=-1 are distinct}
	Let $s\neq r$ and $\rho(g,r,d)=\rho(g,s,e)=-1$, then $\BN{g}{r}{d}$ and $\BN{g}{s}{e}$ are not contained in each other.
\end{theorem}
\begin{proof}
	If $r=1$ or $s=1$, then this follows from \cite[Proposition 1.6]{auel_haburcak_2022}. Thus we may assume $r,s\ge 2$. The proof of \Cref{maxformula} shows that $g+1\le\floor{\frac{d}{r}}+d,\floor{\frac{s}{e}}+s$. The result now follows from \Cref{cor distinguishing loci with same rho}.
\end{proof}

\begin{theorem}\label{theorem rho=-2 loci not contained in BN divisors improvement}
	Suppose $\rho(g,s,e)=-1$, $\rho(g,r,d)=-2$, and $e-2s>d-2r+1$, then $\BN{g}{r}{d}\nsubseteq\BN{g}{s}{e}$.
\end{theorem}
\begin{proof}
	The case $r=1$ follows from \cite[Proposition 1.6]{auel_haburcak_2022} without the assumption that $e-2s>d-2r+1$. Hence we may assume $r\ge 2$, whereby the proof of \Cref{maxformula} implies that $g+1\le \floor{\frac{d}{r}}+d$.
	
	If $s=1$, then $\rho(g,s,e)=-1$ implies that $g+1\le \floor{\frac{e}{s}}+e$. Thus for any $s$ we have $\maxk(g,s,e)=g+1-\gamma(s,e)-2$. The result now follows from \Cref{prop BN locus contianed in BN divisor}.
\end{proof}
\begin{remark}
	We note that this slightly improves the bound from \cite[Corollary 3.6]{CHOI2022}.
\end{remark}

\begin{remark}\label{remark Montserrat examples}
	In \cite[Example 3.2]{bigas2023brillnoether}, potential containments of expected maximal Brill--Noether loci of the form $\BN{g}{r}{d}\subseteq \BN{g}{s}{e}$ with $\rho(g,r,d)=-2$ and $\rho(g,s,e)=-1$ are given. We briefly recall two such examples, and show that \Cref{prop BN locus contianed in BN divisor} does not address these potential containments.
	
	For the potential containment of the form $\BN{2\alpha^2+\alpha-2}{\alpha-1}{2\alpha^2-4}\subseteq \BN{2\alpha^2+\alpha-2}{\alpha}{2\alpha^2-1}$, computing $\maxk$ shows that the Brill--Noether loci both have $\maxk=3\alpha-2$, hence other techniques are required to prove this non-containment.
	
	For the potential containment of the form $\BN{\alpha^2-2}{\alpha-1}{\alpha^2-3}\subseteq \BN{\alpha^2-2}{\alpha-2}{\alpha^2-5}$, computing $\maxk$ shows that $\maxk(\alpha^2-2,\alpha-1,\alpha^2-3)=2\alpha-3,$ while $\maxk(\alpha^2-2,\alpha-2,\alpha^2-5)=2\alpha-2.$ Thus \Cref{Theorem makx>maxk' implies non-containment} gives a non-containment $\BN{\alpha^2-2}{\alpha-2}{\alpha^2-5}\nsubseteq \BN{\alpha^2-2}{\alpha-1}{\alpha^2-3}$ (which already follows for dimension reasons), but cannot show non-containment in the other direction.
\end{remark}

\subsection{Non-containments of expected maximal Brill--Noether loci}\label{subsection noncontainments of exp max BN loci}
For the expected maximal loci $\M^r_g$, we have that $r \leq
\ceil{\sqrt{g}-1}$ by \Cref{rbound}, and that $\maxk$ is
generally a decreasing function of~$r$. However, it is not strictly decreasing (see \Cref{fig1}).
Thus, one expects \Cref{Theorem makx>maxk' implies non-containment} would generally give non-containments $\M^r_g \nsubseteq \M^s_g$ for $r<s$, but to prove such results, we need to control the variation of $\maxk$.

\begin{figure}[h!]
	\centering
	\subfloat[\centering $g=96$]{\includegraphics[width=0.52\linewidth]{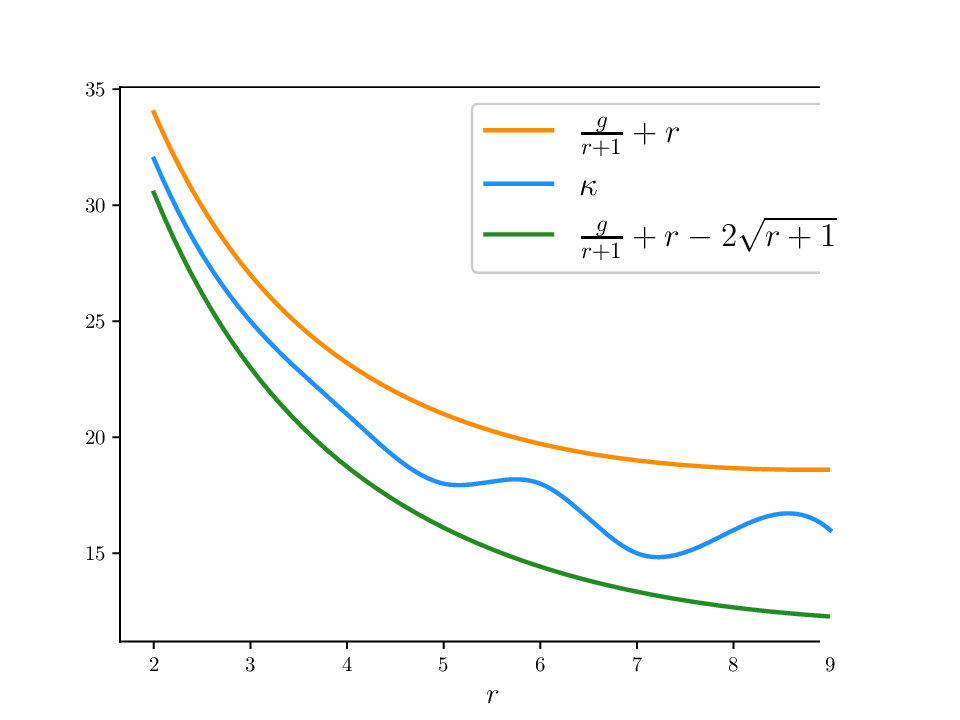}}%
	\qquad
	\hspace{-0.11\linewidth}
	\subfloat[\centering $g=479$]{\includegraphics[width=0.52\linewidth]{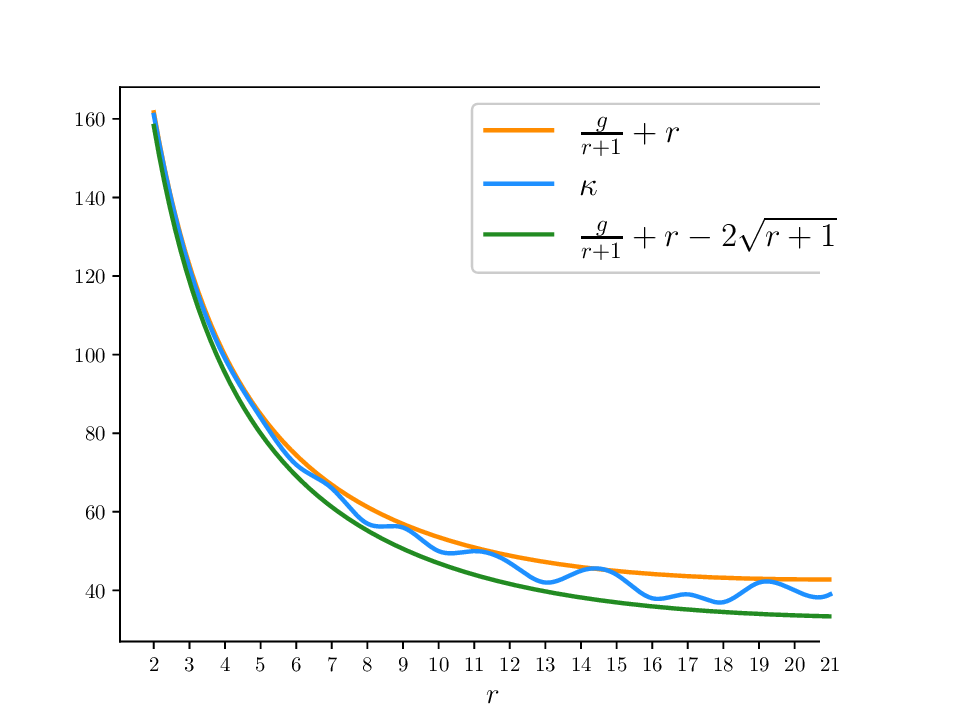}}%
	\caption{Plot of $\maxk(g, r, \dmax(g, r))$.}%
	\label{fig1}
\end{figure}

The first step is to give the following bounds on $\maxk$, pictured by the orange and green curves in \Cref{fig1}.

\begin{lemma}\label{lemma general bounds on maxk}
	For an expected maximal Brill--Noether locus $\BN{g}{r}{d}$, the following inequalities hold.
	\begin{enumerate}[label=(\roman*)]
		\item $\maxk(g,r,d)\le \frac{g}{r+1}+r .$
		\item $\maxk(g,r,d)>  \frac{g}{r+1}+r-2\sqrt{r+1}.$
	\end{enumerate}
\end{lemma}
\begin{proof} When $r = 1$, we have $\maxk(g, 1, \dmax(g,1)) = \dmax(g,1) = \ceil{\frac{g}{2}}$, which satisfies the bounds in $(i)$ and $(ii)$. We thus assume $r \geq 2$.

	To prove $(i)$, we first observe that since $-\rho\ge 1$, we have $-2\sqrt{-\rho}\le -2$.
We also trivially have $\left\lfloor \frac{-gr}{r+1}\right\rfloor \le \frac{-gr}{r+1}$, whence $(i)$ follows from \Cref{maxformula}.
	
	To prove $(ii)$, we make similar observations. We first note that since $ r+1-(g\bmod r+1) \le r+1$, we have $-2\sqrt{r+1 -(g\bmod(r+1))} \ge -2\sqrt{r+1}$, thus \[\left\lfloor -2\sqrt{r+1 -(g\bmod(r+1))} \right\rfloor \ge \left\lfloor -2\sqrt{r+1}\right\rfloor > - 2\sqrt{r+1} - 1.\] Trivially we have $\lfloor \frac{-gr}{r+1} \rfloor> \frac{-gr}{r+1}-1$, whence $(ii)$ follows from \Cref{maxformula}.
\end{proof}

These bounds give rise to the following criterion for non-containments.
\begin{prop}\label{prop general non-containment of maximal BN loci from general bounds}
	Let $\delta\ge 1$. If \[f(g,r,\delta)\colonequals(r+1)\delta^2 +\left( (r+1)(r+1+2\sqrt{r+1})-g\right)\delta + 2(r+1)^2\sqrt{r+1}\le0,\] then $\M^r_g \nsubseteq \M^{r+\delta}_g$.
\end{prop}
\begin{proof}
	The inequality $f(g,r,\delta)\le0$ is equivalent to \[\frac{g}{r+1}+r-2\sqrt{r+1}\ge\frac{g}{r+\delta+1}+r+\delta.\] The result then follows from \Cref{lemma general bounds on maxk} and \Cref{Theorem makx>maxk' implies non-containment}.
\end{proof}

Considering $f(g, r, \delta)$ as a quadratic polynomial in $\delta$, we notice that in the limit of large $g$, the two roots of $f(g, r, \delta)$ tend to $0$ and $g$.
Thus, for $g$ sufficiently large, $f(g, r, \delta) \leq 0$ for all
$1\le\delta \leq \sqrt{g}$. Since expected maximal loci $\M^s_g$ have
$s \leq \ceil{\sqrt{g} - 1} \leq \sqrt{g}$ by \Cref{rbound}, this implies the non-containment $\M^r_g \not\subseteq \M^s_g$ for all $s > r$.
Below we provide an explicit bound on how large $g$ must be in terms of $r$ to achieve all such non-containments.

\begin{theorem}
\label{Theorem one direction of non-containments of max BN loci}
Fix $r\ge 2$. If
\[g \geq 4(r+1)^{5/2} + (r+1)^2 + 2(r+1)^{3/2},\]
then $\M^r_g \nsubseteq \M^s_g$ for all $s > r$.
\end{theorem}
\begin{proof}
Let $\alpha = \sqrt{r+1}$, so that
\[f(g, r, \delta) = \alpha^2\delta^2 + (\alpha^2(\alpha^2 + 2\alpha) - g) \delta + 2\alpha^5.\]
Setting $m = \frac{1}{\alpha^2}(g - \alpha^2(\alpha^2 + 2\alpha))$, we have
\[\frac{1}{\alpha^2} f(g, r, \delta) = \delta^2 - m\delta + 2\alpha^3. \]
Thus, the roots of $f(g, r, \delta)$ are
\begin{align*}
\delta^{\pm} = \frac{1}{2}(m \pm m) \mp \frac{1}{2}m(1 - \sqrt{1 - 8\alpha^3/m^2}).
\end{align*}
By \Cref{prop general non-containment of maximal BN loci from general bounds}, it suffices to show that $\delta^- \leq 1$ and $\delta^+ \geq \sqrt{g} - 1$. Indeed, if so, then $f(g, r, \delta) \leq 0$ for all $\delta \leq \sqrt{g} - 1$, which implies all desired non-containments.

Note that for $0 \leq x \leq 1$ we have $1 - x \leq \sqrt{1-x}$, so $1 - \sqrt{1 - x} \leq x$.
Thus,
\[\frac{1}{2}m(1 - \sqrt{1 - 8\alpha^3/m^2}) \leq  \frac{4\alpha^3}{m}.\]
If $g \geq 4\alpha^5 + \alpha^4 + 2\alpha^3$, then $m \geq 4\alpha^3$. It follows that $\delta^- \leq 1$ and $\delta^+ \geq m - 1$. It thus remains to show that $m \geq \sqrt{g}$, equivalently $m^2 \geq g$, or equivalently
\[g^2 - (3\alpha^4 + 4\alpha^3)g + \alpha^4(\alpha^2 + 2\alpha)^2 \geq 0. \]
The larger root of this quadratic polynomial in $g$ is at
\[\frac{3\alpha^4 + 4\alpha^3 + \sqrt{5\alpha^8 + 8 \alpha^7}}{2},\]
which one readily checks is less than $4\alpha^5 +  \alpha^4 + 2\alpha^3$ for all $\alpha \geq 1$.
\end{proof}

We have now shown that for each $r$, there exists a smallest $G(r)$
such that
\begin{equation} \label{ineq}
\maxk(g, r, \dmax(g, r)) > \maxk(g, s, \dmax(g, s))
\end{equation}
for all $g \geq G(r)$ and $r < s \leq \floor{\sqrt{g}-\frac{1}{2}}$.
\Cref{Theorem one direction of non-containments of max BN loci} gives
an upper bound for $G(r)$, but it is not optimal. Nevertheless, for
any fixed $r$, one can easily check for each of the finitely many $g
\leq 4(r+1)^{5/2} + (r+1)^2 + 2(r+1)^{3/2}$ if \eqref{ineq} holds for
all $s > r$.  We summarize the resulting values of $G(r)$ for low $r$
below.

\medskip
\begin{center}
	\begin{tabular}{|c||c|c|c|c|c|c|c|c|c|}
		\hline
		$r$ & 2 & 3 & 4 & 5 & 6 & 7 & 8 & 9 & 10 \\
		\hline
		 $G(r)$ & 28 & 50 & 96 & 140 & 232 & 306 & 390 & 561 & 684\\
		\hline
	\end{tabular}
\end{center}

\medskip

\begin{remark}
If we fix $r \geq 2$, then there also exist various $g < G(r)$ such
that \eqref{ineq} holds for all $s > r$. For example, \eqref{ineq}
holds for all $s > r$ when
	\begin{itemize}
		\item $r=2$ and $g\notin \{10,	11,	12,	15,	18,	19, 24,	27\}$;
		\item $r=3$ and $g\notin \{17,	18,	19,	21,	24, 28,	29,	33,	34,	41,	44,	49\}$;
		\item $r=4$ and $g\notin \{26, 27,	28,	29,	30,	32,	35,	40,	41,	45,	46,	47,	48,	50,	52,	53,	55,	62,	65,	70,	71,	77,	95\}$.
	\end{itemize}
\end{remark}

\begin{cor}
Except for $g=7,9$, and possibly $g=24,27$, the expected maximal Brill--Noether locus $\M^2_g$ is maximal.
\end{cor}
\begin{proof}
Write $\M^2_g = \M^2_{g,d}$, so that $d = \dmax(g,2)$.  The argument for \cite[Lemma 6.7 (iii)]{auel_haburcak_2022} shows that for a polarized K3 surface $(S,H)$ of genus $g\ge 14$ with Picard group $\Pic(S)=\Z H \oplus \Z L$ with $H.L=d$, and $L^2=2$, a smooth curve $C\in |H|$ has general Clifford index, hence general gonality. Indeed, if $\gamma(C)<\floor{\frac{g-1}{2}}$, one first appeals to \cite[Lemma 8.3]{Knutsen2001} and \cite[Theorem 4.2]{Lelli_Chiesa_2015} and then uses the argument of \cite[Lemma 6.7 (iii)]{auel_haburcak_2022}. We claim that $C$ also has a $g^2_d$, hence $\M^2_g\nsubseteq \M^1_g$. Clearly, $L\vert_C$ is a $g^s_d$ for some $s$, and it suffices to show that $s\ge 2$, as then by adding or subtracting points (as in the trivial containments of Brill--Noether loci), $C$ will have a $g^2_d$. 

To this end, remark that since $L.H>0$, we have $h^2(S,L) = h^0(S,
L^\vee)=0$. Thus applying Riemann--Roch, we see $h^0(S,L)\ge
2+\frac{L^2}{2}\ge 3$. Note also that $(L - H).H < 0$, so $h^0(S, L -
H) =0$. Now taking the long exact sequence in cohomology associated to
the short exact sequence \[0\to L\otimes\calO_S(-C) \to L \to
L\otimes\calO_C \to 0,\] shows that $h^0(C,L\vert_C)\ge h^0(S,L)\ge
3$, whereby $s\ge 2$, as desired. 

Thus $\M^2_g\nsubseteq\M^1_g$ as soon as $g\ge 14$. 
Likewise, as shown in \cite[\S6]{auel_haburcak_2022}, for $6\le g\le 13$, except for $g=7,9$, we also have $\M^2_g\nsubseteq \M^1_g$. As in the above remark, except for possibly $g\in\{10, 11, 12, 15,18, 19, 24, 27\}$, the expected maximal Brill--Noether locus $\M^2_g$ is not contained in any other expected maximal Brill--Noether locus. From \cite{auel_haburcak_2022}, the locus $\M^2_g$ is already known to be maximal when $g=8,10,11,12,15,18,19$.
\end{proof}

\begin{remark}
More generally, a similar argument involving K3 surfaces with $H.L=\dmax(g,r)$ and $L^2=2r-2$ shows that $\M^r_g \nsubseteq \M^1_g$ for $g\ge 14$. 
\end{remark}

\begin{remark}
	In case $\maxk(g,r,\dmax(g,r))=\maxk(g,s,\dmax(g,s))$, other techniques are required to prove non-containments. For example, in genus $24$, $\maxk(24,2,17)=\maxk(24,4,23)$. Since $\rho(24,2,17)=-3$ and $\rho(24,4,23)=-1$, we have a non-containment $\BN{24}{4}{23}\nsubseteq\BN{24}{2}{17}$ for dimension reasons. The reverse containment is unknown.
	
	Similarly, in genus $27$, $\maxk(27,2,19)=\maxk(27,3,23)$. As $\rho(27,2,19)=-3$ and $\rho(27,3,23)=-1$, we have a non-containment $\BN{27}{3}{23}\nsubseteq \BN{27}{2}{19}$. The reverse non-containment is unknown.
\end{remark}

\bigskip

	%%%%Before posting, comment out the \bib stuff, copy the .bbl into the .tex, and then use \input instead%%%
%\bibliographystyle{amsplain}
%\bibliography{bib}

\begin{thebibliography}{10}
	
	\bibitem{auel_haburcak_2022}
	Asher Auel and Richard Haburcak, \emph{{M}aximal {B}rill--{N}oether loci via
		{K}3 surfaces}, arXiv:2206.04610, 2022.
	
	\bibitem{bud2024brillnoether}
	Andrei Bud, \emph{{B}rill--{N}oether loci and strata of differentials}, 2024.
	
	\bibitem{bh_2024maximal}
	Andrei Bud and Richard Haburcak, \emph{Maximal {B}rill--{N}oether loci via
		degenerations and double covers}, 2024.
	
	\bibitem{CHOI2022}
	Youngook Choi and Seonja Kim, \emph{Linear series on a curve of compact type
		bridged by a chain of elliptic curves}, Indagationes Mathematicae (2022),
	844--860.
	
	\bibitem{CHOI2012377}
	Youngook Choi, Seonja Kim, and Young~Rock Kim, \emph{Remarks on
		{B}rill--{N}oether divisors and {H}ilbert schemes}, Journal of Pure and
	Applied Algebra \textbf{216} (2012), no.~2, 377--384.
	
	\bibitem{CHOI20141458}
	\bysame, \emph{{B}rill--{N}oether divisors for even genus}, Journal of Pure and
	Applied Algebra \textbf{218} (2014), no.~8, 1458--1462.
	
	\bibitem{Powell_jensen_fixed_gonality}
	Kaelin Cook-Powell and David Jensen, \emph{Components of {B}rill--{N}oether
		loci for curves with fixed gonality}, 2019.
	
	\bibitem{cook-powell_jensen}
	\bysame, \emph{Components of {B}rill--{N}oether loci for curves with fixed
		gonality}, Michigan Math. J. \textbf{71} (2022), no.~1, 19--45.
	
	\bibitem{eisenbud_harris}
	David Eisenbud and Joe Harris, \emph{The {K}odaira dimension of the moduli
		space of curves of genus {$\ge$} 23}, Inventiones mathematicae \textbf{90}
	(1987), no.~2, 359--387.
	
	\bibitem{Eisenbud_Harris_1989}
	\bysame, \emph{Irreducibility of some families of linear series with
		{B}rill--{N}oether number {$ -1 $}}, Annales scientifiques de l'\'Ecole
	Normale Sup\'erieure \textbf{Ser. 4, 22} (1989), no.~1, 33--53 (en).
	
	\bibitem{Farkas2000}
	Gavril Farkas, \emph{The geometry of the moduli space of curves of genus 23},
	Mathematische Annalen \textbf{318} (2000), no.~1, 43--65.
	
	\bibitem{Farkas_2001}
	\bysame, \emph{{B}rill--{N}oether loci and the gonality stratification of
		{$\mathcal{M}_g$}}, {J}. Reine. Angew. {M}ath. \textbf{2001} (2001), no.~539,
	185--200.
	
	\bibitem{gieseker}
	D.~Gieseker, \emph{Stable curves and special divisors: {P}etri's conjecture},
	Invent. Math. \textbf{66} (1982), no.~2, 251--275.
	
	\bibitem{griffiths_harris}
	Phillip Griffiths and Joseph Harris, \emph{On the variety of special linear
		systems on a general algebraic curve}, Duke Math. J. \textbf{47} (1980),
	no.~1, 233--272.
	
	\bibitem{harris_mumford}
	Joe Harris and David Mumford, \emph{On the {K}odaira dimension of the moduli
		space of curves}, Invent. Math. \textbf{67} (1982), no.~1, 23--86.
	
	\bibitem{jensen_payne_2021recent}
	David Jensen and Sam Payne, \emph{Recent developments in {B}rill--{N}oether
		theory}, 2021, To appear in EMS Volume for the BMS Thematic Einstein Semester
	on Algebraic Geometry, arXiv:2111.00351.
	
	\bibitem{jensen_ranganathan}
	David Jensen and Dhruv Ranganathan, \emph{Brill--{N}oether theory for curves of
		a fixed gonality}, Forum Math. Pi \textbf{9} (2021), Paper No. e1, 33.
	
	\bibitem{Knutsen2001}
	Andreas~Leopold Knutsen, \emph{On kth-order embeddings of {K}3 surfaces and
		{E}nriques surfaces}, Manuscripta Math. \textbf{104} (2001), no.~2, 211--237.
	
	\bibitem{larson_larson_vogt_2020global}
	Eric Larson, Hannah Larson, and Isabel Vogt, \emph{Global {B}rill--{N}oether
		theory over the {H}urwitz space}, 2020, to appear in Geom.\ Topol.,
	arXiv:2008.10765.
	
	\bibitem{larson_refined_BN_Hurwitz}
	Hannah~K. Larson, \emph{A refined {B}rill--{N}oether theory over {H}urwitz
		spaces}, Invent. Math. \textbf{224} (2021), no.~3, 767--790.
	
	\bibitem{lazarsfeld:Brill-Noether_without_degenerations}
	Robert Lazarsfeld, \emph{Brill--{N}oether--{P}etri without degenerations}, J.
	Differ. Geom. \textbf{23} (1986), no.~3, 299--307.
	
	\bibitem{Lelli-Chiesa_the_gieseker_petri_divisor_g_le_13}
	Margherita Lelli-Chiesa, \emph{The {G}ieseker--{P}etri divisor in {$M_g$} for
		{$g\leq 13$}}, Geom. Dedicata \textbf{158} (2012), 149--165.
	
	\bibitem{Lelli_Chiesa_2015}
	\bysame, \emph{Generalized {L}azarsfeld--{M}ukai bundles and a conjecture of
		{D}onagi and {M}orrison}, Adv. Math. \textbf{268} (2015), 529--563.
	
	\bibitem{bigas2023brillnoether}
	{M}ontserrat {T}eixidor~i {B}igas, \emph{{B}rill--{N}oether loci}, 2023,
	arXiv:2308.10581.
	
	\bibitem{Mukai_Curves_and_grassmannians_1993}
	Shigeru Mukai, \emph{Curves and {G}rassmannians}, Algebraic geometry and
	related topics ({I}nchon, 1992), Conf. Proc. Lecture Notes Algebraic Geom.,
	I, Int. Press, Cambridge, MA, 1993, pp.~19--40.
	
	\bibitem{pflueger}
	Nathan Pflueger, \emph{Brill--{N}oether varieties of {$k$}-gonal curves}, Adv.
	Math. \textbf{312} (2017), 46--63.
	
	\bibitem{steffen_1998}
	Frauke Steffen, \emph{A generalized principal ideal theorem with an application
		to {B}rill--{N}oether theory}, Invent. Math. \textbf{132} (1998), no.~1,
	73--89.
	
\end{thebibliography}
\providecommand{\bysame}{\leavevmode\hbox to3em{\hrulefill}\thinspace}
\providecommand{\MR}{\relax\ifhmode\unskip\space\fi MR }
\providecommand{\MRhref}[2]{%
	\href{http://www.ams.org/mathscinet-getitem?mr=#1}{#2}
}
\providecommand{\href}[2]{#2}

\vfill

\end{document}